\documentclass{article}
\usepackage {amssymb}
\usepackage {amsmath}
\usepackage {amsmath}
\usepackage {amsthm}
\usepackage{graphicx}
\usepackage {amscd}
\usepackage[colorlinks,linkcolor=black, anchorcolor=black, citecolor=red]{hyperref}
\newtheorem{thm}{Theorem}
\newtheorem{prop}{Proposition}
\newtheorem{rem}{Remark}
\newtheorem{exmp}{Example}
\newtheorem{lem}{Lemma}
\newtheorem{claim}{Claim}

\begin{document}
\title{\textbf{Greatest lower bounds on Ricci curvature for toric Fano
manifolds}}
\author{Chi Li}
\date{}
\maketitle
\begin{abstract}
\noindent ABSTRACT: In this short note, based on the work of
Wang-Zhu \cite{WZ}, we determine the greatest lower bounds on Ricci
curvature for all toric Fano manifolds.
\end{abstract}
\begin{section}{Introduction}
On Fano manifolds $X$, i.e. $K_X^{-1}$ is ample, the
K\"{a}hler-Einstein equation
\[
Ric(\omega)=\omega
\]
is equivalent to the complex Monge-Amp\`{e}re equation:
\begin{align}
(\omega+\partial\bar{\partial}\phi)^n=e^{h_\omega-\phi}\omega\tag*{$(*)$}\label{KEeq}
\end{align}
where $\omega$ is a fixed K\"{a}hler metric in $c_1(X)$, and
$h_\omega$ is the normalized Ricci potential:
\begin{equation}\label{ricpot}
Ric(\omega)-\omega=\partial\bar{\partial}h_\omega,\;\;\int_Xe^{h_\omega}\omega^n=\int_X\omega^n
\end{equation}
In order to solve this equation, the continuity method is used. So
we consider a family of equations with parameter $t$:
\begin{align}
(\omega+\partial\bar{\partial}\phi_t)^n=e^{h_\omega-t\phi}\omega^n\tag*{$(*)_t$}\label{CMAt}
\end{align}
Define $S_t=\{t: \ref{CMAt} \mbox{ is solvable}\}$. It was known
that the set $S_t$ is open. To solve \ref{KEeq}, the crucial thing
is to obtain the closedness of this set. So we need some a prior
estimates. By Yau's $C^2$ and Calabi's higher order estimates(See
\cite{Yau}, \cite{T}), we only need uniform $C^0$-estimates for
solutions $\phi_t$ of \ref{CMAt}. In general one can not solve
\ref{KEeq}, and so can not get the $C^0$-estimates, due to the well
known obstruction of Futaki invariant. So when $t\rightarrow R(X)$,
some blow-up happens.

It was first showed by Tian \cite{T2} that we may not be able to
solve \ref{CMAt} on certain Fano manifold for $t$ sufficiently close
to 1. Equivalently, for such a Fano manifold, there is some $t_0<1$,
such that there is no K\"{a}hler metric $\omega$ in $c_1(X)$ which
can have $Ric(\omega)\ge t_0\omega$. It is now made more precise.

Define
\[
R(X)=sup\{t:\ref{CMAt} \mbox{ is solvable}\}
\]
It can be shown that $R(X)$ is independent of $\omega\in c_1(X)$. In
fact, Sz\'{e}kelyhidi \cite{S} observed 
\\[12.5pt]
\textbf{Fact}: $\quad R(X)=\sup\{t: Ric(\omega)>t\omega,\forall\;
\mbox{K\"{a}hler metric}\; \omega\in c_1(X)\} $\\[12.5pt]
He also showed $R(Bl_p\mathbb{P}^2)=\frac{6}{7}$ and $\frac{1}{2}\le
R(Bl_{p,q}\mathbb{P}^2)\le\frac{21}{25}$.

Let $\Lambda\simeq \mathbb{Z}^n$ be a lattice in
$\mathbb{R}^n=\Lambda\otimes_{\mathbb{Z}}\mathbb{R}$. A toric Fano
manifold $X_\triangle$ is determined by a reflexive lattice polytope
$\triangle$ (For details on toric manifolds, see \cite{Oda}). For
example, the toric manifold $Bl_p\mathbb{P}^2$ is determined by the
following polytope.
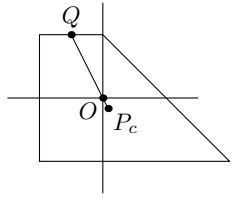
\begin{center}\label{figure1}
\includegraphics[width=0.20\textwidth]{toric1-1.1}\hspace{0.2\textwidth}
\end{center}
In this short note, we determine $R(X_\triangle)$ for every toric
Fano manifold $X_\triangle$ in terms of the geometry of polytope
$\triangle$.

Any such polytope $\triangle$ contains the origin
$O\in\mathbb{R}^n$. We denote the barycenter of $\triangle$ by
$P_c$. If $P_c\neq O$, the ray
$P_c+\mathbb{R}_{\ge0}\cdot\overrightarrow{P_c O}$ intersects the
boundary $\partial\triangle$ at point $Q$. Our main result is
\begin{thm}
If $P_c\neq O$,
\[
R(X_\triangle)=\frac{|\overline{OQ}|}{|\overline{P_cQ}|}
\]
Here $|\overline{OQ}|$, $|\overline{P_cQ}|$ are lengths of line
segments $\overline{OQ}$ and $|\overline{P_cQ}|$. If $P_c=O$, then
there is K\"{a}hler-Einstein metric on $X_\triangle$ and
$R(X_\triangle)=1$.
\end{thm}
\begin{rem}
Note for the toric Fano manifold, $P_c$ is just Futaki invariant. So
the second statement follows from Wang-Zhu \cite{WZ}. We will repeat
the proof in next section.
\end{rem}
Our method is based on Wang-Zhu's \cite{WZ} theory for proving the
existence of K\"{a}hler-Ricci solitons on toric Fano manifolds. In
view of the analysis in \cite{WZ}, if $R(X_\triangle)<1$, then as
$t\rightarrow R(X_\triangle)$, the blow-up happens exactly because
the minimal points of a family of proper convex functions go to
infinity, or, equivalently, the images of minimal points under the
momentum map of a fixed metric tend to the boundary of the toric
polytope. The key identity relation in [Section 2,\eqref{keyiden1}]
and some uniform a priori estimates enable us to read out
$R(X_\triangle)$ in terms of geometry of $\triangle$.

This note is partly inspired by the Sz\'{e}kelyhidi's paper \cite{S}
and Donaldson's survey \cite{D}. The author thanks Professor Gang
Tian for constant encouragement.
\end{section}
\begin{section}{Consequence of Wang-Zhu's theory}
First we recall the set up of Wang-Zhu \cite{WZ}. For a reflexive
lattice polytope $\triangle$ in
$\mathbb{R}^n=\Lambda\otimes_\mathbb{Z}\mathbb{R}$, we have a Fano
toric manifold $(\mathbb{C}^*)^n\subset X_\triangle$ with a
$(\mathbb{C}^*)^n$ action. Let $\{z_i\}$ be the standard coordinates
of the dense orbit $(\mathbb{C}^*)^n$, and $x_i=\log|z_i|^2$. Let
$\{p_\alpha\}_{\alpha=1,\cdots,N}$ be the lattice points contained
in $\triangle$. We take the fixed K\"{a}hler metric $\omega$ to be
given by the potential (on $(\mathbb{C}^*)^n$)
\begin{equation}\label{u0}
\tilde{u}_0=\log\left(\sum_{\alpha=1}^N e^{<p_\alpha,x>}\right)+C
\end{equation}
$C$ is some constant determined by normalization condition:
\begin{equation}\label{normalize}
\int_{\mathbb{R}^n}e^{-\tilde{u}_0}dx=Vol(\triangle)=\frac{1}{n!}\int_{X_\triangle}\omega^n=\frac{c_1(X_\triangle)^n}{n!}
\end{equation}
By standard toric geometry, each lattice point $p_\alpha$ contained
in $\triangle$ determines, up to a constant, a
($\mathbb{C}^*)^n$-equivariant section $s_\alpha$ in
$H^0(X,K^{-1}_X)$.  We can embed $X_\triangle$ into
$P(H^0(X,K^{-1}_X)^*)$ using these sections. Let $s_0$ be the
section corresponding to the origin $0\in\triangle$, then its
Fubini-Study norm is
\[
|s_0|_{FS}^2=\frac{|s_0|^2}{\sum_{\alpha=1}^N|s_\alpha|^2}=\left(\sum_{\alpha=1}^N\prod_{i=1}^n
|z_i|^{2p_{\alpha,i}}\right)^{-1}=\left(\sum_{\alpha=1}^N
e^{<p_\alpha,x>}\right)^{-1}=e^{C}e^{-\tilde{u}_0}
\]
So the K\"{a}hler metric
$\omega=\frac{\sqrt{-1}}{2\pi}\partial\bar{\partial}\tilde{u}_0$ is
the Fubini-Study metric.

On the other hand, $Ric(\omega)$ is the curvature of Hermitian line
bundle $K_{M}^{-1}$ with Hermitian metric determined by the volume
form $\omega^n$. Note that on the open dense orbit
$(\mathbb{C}^*)^n$, we can take $s_0=z_1\frac{\partial}{\partial
z_1}\wedge\cdots\wedge z_n\frac{\partial}{\partial z_n}$. Since
$\frac{\partial}{\partial\log
z_i}=\frac{1}{2}(\frac{\partial}{\partial\log|z_i|}-\sqrt{-1}\frac{\partial}{\partial\theta_i})=\frac{\partial}{\partial\log|z_i|^2}=\frac{\partial}{\partial
x_i}$ when acting on any $(S^1)^n$ invariant function on
$(\mathbb{C}^*)^n$, we have
\begin{eqnarray*}
|s_0|_{\omega^n}^2&=&\left|z_1\frac{\partial}{\partial
z_1}\wedge\cdots\wedge z_n\frac{\partial}{\partial
z_n}\right|_{\omega^n}^2
=\det\left(\frac{\partial^2\tilde{u}_0}{\partial\log
z_i\;\overline{\partial\log
z_j}}\right)\\
&=&\det\left(\frac{\partial^2\tilde{u}_0}{\partial\log
|z_i|^2\;\partial\log|z_j|^2}\right)=\det(\tilde{u}_{0,ij})
\end{eqnarray*}

It's easy to see from definition of $h_\omega$ \eqref{ricpot} and
normalization condition \eqref{normalize} that
\[e^{h_\omega}=e^{-C}\frac{|s_0|_{FS}^2}{|s_0|_{\omega^n}^2}=e^{-\tilde{u}_0}\det(\tilde{u}_{0,ij})^{-1}\]

Then using the torus symmetry, \ref{CMAt} can be translated into
real Monge-Amp\`{e}re equation \cite{WZ} on $\mathbb{R}^n$.
\begin{align}
\det(u_{ij})=e^{-(1-t)\tilde{u}_0-tu}=e^{-w_t}\tag*{$(**)_t$}\label{RMAt}
\end{align}
The solution $u_t$ of \ref{RMAt} is
related to K\"{a}hler potential $\phi_t$ in \ref{CMAt} by the
identity:
\begin{equation}
u=\tilde{u}_0+\phi_t
\end{equation}
where $\phi_t$ is viewed as a function of $x_i=\log|z_i|^2$ by torus
symmetry.

Every strictly convex function $f$ appearing in \ref{RMAt}
($f=\tilde{u}_0$, $u$, $w_t=(1-t)\tilde{u}_0+tu_t$) must satisfy
$Df(\mathbb{R}^n)=\triangle^\circ$ ($\triangle^\circ$ means the
interior of $\triangle$). Since $0$ is (the unique lattice point)
contained in $\triangle^\circ=Df(\mathbb{R}^n)$, the strictly convex
function $f$ is properly.

Wang-Zhu's \cite{WZ} method for solving \ref{RMAt} consists of two
steps. The \textbf{first step} is to show some uniform a priori
estimates for $w_t$. For $t<R(X_\triangle)$, the proper convex
function $w_t$ obtains its minimum value at a unique point
$x_t\in\mathbb{R}^n$. Let
\[
m_t=inf\{w_t(x):x\in\mathbb{R}^n\}=w_t(x_t)
\]
\begin{prop}[\cite{WZ},See also \cite{D}]
\begin{enumerate}
\item
there exists a constant $C$, independent of $t<R(X_\triangle)$, such
that
\[|m_t|<C\]
\item There exists $\kappa>0$ and a constant $C$, both independent of
$t<R(X_\triangle)$, such that
\begin{equation}\label{esti}
w_t\ge\kappa|x-x_t|-C
\end{equation}
\end{enumerate}
\end{prop}
For the reader's convenience, we record the proof here.
\begin{proof}
Let $A=\{x\in\mathbb{R}^n;m_t\le w(x)\le m_t+1\}$. $A$ is a convex
set. By a well known lemma due to Fritz John, there is a unique
ellipsoid $E$ of minimum volume among all the ellipsoids containing
$A$, and a constant $\alpha_n$ depending only on dimension, such
that
\[
\alpha_n E\subset A\subset E
\]
$\alpha_nE$ means the $\alpha_n$-dilation of $E$ with respect to its
center. Let $T$ be an affine transformation with $\det(T)=1$, which
leaves $x'$=the center of $E$ invariant, such that $T(E)=B(x',R)$,
where $B(x',R)$ is the Euclidean ball of radius $R$. Then
\[
B(x',\alpha_n R)\subset T(A)\subset B(x',R)
\]
We first need to bound $R$ in terms of $m_t$. Since
$D^2w=tD^2u+(1-t)D^2\tilde{u}_0\ge tD^2u$, by \eqref{RMAt}, we see
that
\[
\det(w_{ij})\ge t^ne^{-w}
\]
Restrict to the subset $A$, it's easy to get
\[
\det(w_{ij})\ge C_1e^{-m_t}
\]
Let $\tilde{w}(x)=w(T^{-1}x)$, since $\det(T)=1$, $\tilde{w}$
satisfies the same inequality
\[
\det(\tilde{w}_{ij})\ge C_1e^{-m_t}
\]
in $T(A)$.

Construct an auxiliary function
\[
v(x)=C_1^{\frac{1}{n}}e^{-\frac{m_t}{n}}\frac{1}{2}\left(|x-x'|^2-(\alpha_n
R)^2\right)+m_t+1
\]
Then in $B(x',\alpha_n R)$,
\[
\det(v_{ij})=C_1e^{-m_t}\le \det(\tilde{w}_{ij})
\]
On the boundary $\partial B(x',\alpha_n R)$, $v(x)=m_t+1\ge
\tilde{w}$. By the Bedford-Taylor comparison principle for
Monge-Am\`{e}re operator, we have
\[
\tilde{w}(x)\le v(x)\;\;\mbox{in}\;\; B(x',\alpha_n R)
\]
In particular
\[
m_t\le\tilde{w}(x')\le
v(x')=C_1^{\frac{1}{n}}e^{-\frac{m_t}{n}}\frac{1}{2}(-\frac{R^2}{n^2})+m_t+1
\]
So we get the bound for $R$:
\[
R\le C_2e^{\frac{m_t}{2n}}
\]
So we get the upper bound for the volume of $A$:
\[
Vol(A)=Vol(T(A))\le C R^n\le Ce^{\frac{m_t}{2}}
\]
By the convexity of $w$, it's easy to see that $\{x;w(x)\le
m_t+s\}\subset s\cdot\{x;w(x)\le m_t+1\}=s\cdot A$, where $s\cdot A$
is the $s$-dilation of $A$ with respect to point $x_t$. So
\begin{equation}\label{uppervol}
Vol(\{x;w(x)\le m_t+s\})\le s^n Vol(A)\le Cs^ne^{\frac{m_t}{2}}
\end{equation}
The lower bound for volume of sublevel sets is easier to get.
Indeed, since $|Dw(x)|\le L$, where $L=\max_{y\in\triangle}|y|$, we
have $B(x_t, s\cdot L^{-1})\subset\{x;w(x)\le m_t+s\}$. So
\begin{equation}\label{lowervol}
Vol(\{x;w(x)\le m_t+s\})\ge C s^n
\end{equation}
Now we can derive the estimate for $m_t$. First note the identity:
\begin{equation}\label{vol}
\int_{\mathbb{R}^n}e^{-w}dx=\int_{\mathbb{R}^n}\det(u_{ij})dx=\int_{\triangle}d\sigma=Vol(\triangle)
\end{equation}
Second, we use the coarea formula
\begin{eqnarray}\label{coarea}
\int_{\mathbb{R}^n}e^{-w}dx&=&\int_{\mathbb{R}^n}\int_{w}^{+\infty}e^{-s}ds
dx=\int_{-\infty}^{+\infty}e^{-s}ds\int_{\mathbb{R}^n}1_{\{w\le
s\}}dx=\int_{m_t}^{+\infty}e^{-s}Vol(\{w\le s\})ds\nonumber\\
&=&e^{-m_t}\int_{0}^{+\infty}e^{-s}Vol(\{w\le m_t+s\})ds
\end{eqnarray}
Using the bound for the volume of sublevel sets \eqref{uppervol} and
\eqref{lowervol} in \eqref{coarea}, and compare with \eqref{vol},
it's easy to get the bound for $|m_t|$.

Now we prove the estimate \eqref{esti} following the argument of
\cite{D}. We have seen $B(x_t, L^{-1})\subset\{w\le m_t+1\}$, and
$Vol(\{w\le m_t+1\})\le C$ by \eqref{uppervol} and uniform bound for
$m_t$. Then we must have $\{w\le m_t+1\}\subset B(x_t, R(C,L)) $ for
some uniformly bounded radius $R(C,L)$. Otherwise, the convex set
$\{w\le m_t+1\}$ would contain a convex subset of arbitrarily large
volume. By the convexity of $w$, we have
$
w(x)\ge \frac{1}{R(C,L)}|x-x_t|+m_t-1
$
Since $m_t$ is uniformly bounded, the estimate \eqref{esti} follows.
\end{proof}
The \textbf{second step} is trying to bound $|x_t|$. In Wang-Zhu's
\cite{WZ} paper, they proved the existence of K\"{a}hler-Ricci
soliton on toric Fano manifold by solving the real Monge-Amp\`{e}re
equation corresponding to K\"{a}hler-Ricci solition equation. But
now we only consider the K\"{a}hler-Einstein equation, which in
general can't be solved because there is the obstruction of Futaki
invariant.
\begin{prop}[\cite{WZ}]
the uniform bound of $|x_t|$ for any $0\le t\le t_0$, is equivalent
to that we can solve \ref{RMAt}, or equivalently solve \ref{CMAt},
for $t$ up to $t_0$. More precisely, (by the discussion in
introduction,) this condition is equivalent to the uniform
$C^0$-estimates for the solution $\phi_t$ in \ref{CMAt} for $t\in
[0,t_0]$.
\end{prop}
Again we sketch the proof here.
\begin{proof}
If we can solve \ref{RMAt} (or equivalently \ref{CMAt}) for $0\le
t\le t_0$. Then $\{w(t)=(1-t)\tilde{u_0}+t u; 0\le t\le t_0\}$ is a
smooth family of proper convex functions on $\mathbb{R}^n$. So their
minimal points are uniformly bounded in a compact set.

Conversely, assume $|x_t|$ is bounded. First note that
$\phi_t=u-\tilde{u}_0=\frac{1}{t}(w_t(x)-\tilde{u}_0)$.

As in Wang-Zhu \cite{WZ}, we consider the enveloping function:
\[
v(x)=\max_{p_\alpha\in\Lambda\cap\triangle}\langle p_\alpha,
x\rangle
\]
Then $0\le\tilde{u}_0(x)-v(x)\le C$, and $Dw(\xi)\cdot x\le v(x)$
for all $\xi, x\in\mathbb{R}^n$. We can assume $t\ge\delta>0$. Then
using uniform boundedness of $|x_t|$
\begin{eqnarray*}
\phi_t(x)&=&\frac{1}{t}(w_t(x)-\tilde{u}_0)=\frac{1}{t}[(w_t(x)-w_t(x_t))-v(x)+(v(x)-\tilde{u}_0(x))-w_t(x_t)]\\
&\le& \delta^{-1}(Dw_t(\xi)\cdot x-v(x)-Dw_t(\xi)\cdot x_t)-C\le C'
\end{eqnarray*}
Thus we get the estimate for $\sup_t\phi_t$. Then one can get the
bound for $\inf_t\phi_t$ using the Harnack inequality in the theory
of Monge-Amp\`{e}re equations. For details see (\cite{WZ}, Lemma
3.5) (see also \cite{T3}).
\end{proof}
By the above proposition, we have
\begin{lem}
If $R(X_\triangle)<1$, then there exists a subsequence $\{x_{t_i}\}$
of $\{x_t\}$, such that
\[\lim_{t_i\rightarrow R(X_\triangle)}|x_{t_i}|=+\infty\]
\end{lem}
The observation now is that
\begin{lem}
If $R(X_\triangle)<1$, then there exists a subsequence of
$\{x_{t_i}\}$ which we still denote by $\{x_{t_i}\}$, and
$y_\infty\in
\partial\triangle$, such that
\begin{equation}\label{cvyinf}
\lim_{t_i\rightarrow R(X_\triangle)}D\tilde{u}_0(x_{t_i})=y_\infty
\end{equation}
\end{lem}
This follows easily from the properness of $\tilde{u}_0$ and
compactness of $\triangle$.

We now use the key relation (See \cite{WZ} Lemma 3.3, and also
\cite{D} page 29)
\[
0=\int_{\mathbb{R}^n}Dw(x)e^{-w}dx=\int_{\mathbb{R}^n}((1-t)D\tilde{u}_0+tDu)e^{-w}dx
\]
Since
\[
\int_{\mathbb{R}^n}Du\,e^{-w}dx=\int_{\mathbb{R}^n}Du\det(u_{ij})dx
=\int_{\triangle}yd\sigma=Vol(\triangle)P_c
\]
where $P_c$ is the barycenter of $\triangle$, so
\begin{equation}\label{keyiden1}
\frac{1}{Vol(\triangle)}\int_{\mathbb{R}^n}D\tilde{u}_0e^{-w}dx=-\frac{t}{1-t}P_c
\end{equation}
We will show this vector tend to a point on $\partial\triangle$ when
$t$ goes to $R(X_\triangle)$. To prove this we use the defining
function of $\triangle$. Similar argument was given in the survey
\cite{D}, page 30.
\end{section}
\begin{section}{Proof of Theorem 1}
We now assume the reflexive polytope $\triangle$ is defined by
inequalities:
\begin{equation}\label{deftri}
\lambda_r(y)\ge-1,\;r=1,\cdots,K
\end{equation}
$\lambda_r(y)=\langle v_r,y\rangle$ are fixed linear functions. We
also identify the minimal face of $\triangle$ where $y_\infty$ lies:
\begin{eqnarray}\label{yinface}
&&\lambda_r(y_{\infty})=-1,\; r=1,\cdots,K_0\\
&&\lambda_r(y_{\infty})>-1,\; r=K_0+1,\cdots, K\nonumber
\end{eqnarray}
Clearly, Theorem 1 follows from
\begin{prop}If $P_c\neq O$,
\[-\frac{R(X_\triangle)}{1-R(X_\triangle)}P_c\in\partial\triangle\]
Precisely,
\begin{equation}\label{result}
\lambda_r\left(-\frac{R(X_\triangle)}{1-R(X_\triangle)}P_c\right)\ge
-1
\end{equation}
Equality holds if and only if $r=1,\cdots,K_0$. So
$-\frac{R(X_\triangle)}{1-R(X_\triangle)}P_c$ and $y_\infty$ lie on
the same faces \eqref{yinface}.
\end{prop}
\begin{proof}
By \eqref{keyiden1} and defining function of $\triangle$, we have
\begin{equation}\label{keyiden}
\lambda_r\left(-\frac{t}{1-t}P_c\right)+1=\frac{1}{Vol(\triangle)}\int_{\mathbb{R}^n}\lambda_r(D\tilde{u}_0)e^{-w}dx+1=
\frac{1}{Vol(\triangle)}\int_{\mathbb{R}^n}(\lambda_r(D\tilde{u}_0)+1)e^{-w}dx
\end{equation}
The inequality \eqref{result} follows from \eqref{keyiden} by
letting $t\rightarrow R(X_\triangle)$. To prove the second
statement, by \eqref{keyiden} we need to show
\begin{equation}\label{intcv}
\lim_{t_i\rightarrow
R(X_\triangle)}\frac{1}{Vol(\triangle)}\int_{\mathbb{R}^n}\lambda_r(D\tilde{u}_0)e^{-w_{t_i}}dx+1\left\{
\begin{array}{l@{\quad:\quad}l}
=0& r=1,\cdots,K_0\\
>0& r=K_0+1,\cdots, N  \end{array} \right.
\end{equation}
By the uniform estimate \eqref{esti} and fixed volume \eqref{vol},
and since $D\tilde{u}_0(\mathbb{R}^n)=\triangle^\circ$ is a bounded
set, there exists $R_\epsilon$, independent of $t\in
[0,R(X_\triangle))$, such that
\begin{equation}\label{small}
\frac{1}{Vol(\triangle)}\int_{\mathbb{R}^n \setminus
B_{R_\epsilon}(x_t)}\lambda_r(D\tilde{u}_0)e^{-w_t}dx<\epsilon,\;
and\;\; \frac{1}{Vol(\triangle)}\int_{\mathbb{R}^n \setminus
B_{R_\epsilon}(x_t)}e^{-w_t}dx<\epsilon
\end{equation}
Now (\ref{intcv}) follows from the following claim.
\begin{claim}
Let $R>0$, there exists a constant $C>0$, which only depends on the
polytope $\triangle$, such that for all $\delta x\in
B_{R}(0)\subset\mathbb{R}^n$,
\begin{equation}\label{compare1}
e^{-CR}(\lambda_r(D\tilde{u}_0(x_{t_i}))+1)\le\lambda_r(D\tilde{u}_0(x_{t_i}+\delta
x))+1\le e^{CR}(\lambda_r(D\tilde{u}_0(x_{t_i}))+1)
\end{equation}
\end{claim}
Assuming the claim, we can prove two cases of \eqref{intcv}. First
by (\ref{cvyinf}) and (\ref{yinface}), we have
\begin{equation}\label{conv1}
\lim_{t_i\rightarrow
R(X_\triangle)}\lambda_r(D\tilde{u}_0(x_{t_i}))+1=\lambda_r(y_\infty)+1=\left\{\begin{array}
{l@{\quad:\quad}l} 0& r=1,\cdots,K_0\\ a_r>0& r=K_0+1,\cdots, N
\end{array}\right.
\end{equation}
\begin{enumerate}
\item $r=1,\cdots,K_0$. $\forall\epsilon>0$, first choose $R_\epsilon$ as in \eqref{small}. By (\ref{compare1}) and (\ref{conv1}), there exists $\rho_\epsilon>0$, such that if
$|t_i-R(X_\triangle)|<\rho_\epsilon$, then for all $\delta x\in
B_{R_\epsilon}(0)\subset\mathbb{R}^n$,
\[
0\le\lambda_r(D\tilde{u}_0(x_{t_i}+\delta x))+1<e^{C
R_\epsilon}(\lambda_r(D\tilde{u}_0)(x_{t_i})+1)<\epsilon
\]
in other words, $\lambda_r(D\tilde{u}_0(x_{t_i}+\delta
x))+1\rightarrow 0$ uniformly for $\delta x \in B_{R_\epsilon}(0)$,
as $t_i\rightarrow R(X_\triangle)$.
So when $|t_i-R(X_\triangle)|<\rho_\epsilon$,
\begin{eqnarray*}
\frac{1}{Vol(\triangle)}\int_{\mathbb{R}^n}\lambda_r(D\tilde{u}_0)e^{-w}dx+1&=&\frac{1}{Vol(\triangle)}\int_{\mathbb{R}^n
\setminus
B_{R_\epsilon}(x_{t_i})}\lambda_r(D\tilde{u}_0)e^{-w}dx+\frac{1}{Vol(\triangle)}\int_{\mathbb{R}^n
\setminus
B_{R_\epsilon}(x_{t_i})}e^{-w}dx\\
&&+\frac{1}{Vol(\triangle)}\int_{B_{R_\epsilon}(x_{t_i})}(\lambda_r(D\tilde{u}_0)+1)e^{-w}dx\\
&\le&2\epsilon+\epsilon\frac{1}{Vol(\triangle)}\int_{B_{R_\epsilon}(x_{t_i})}e^{-w}dx\le
3\epsilon
\end{eqnarray*}
The first case in (\ref{intcv}) follows by letting
$\epsilon\rightarrow 0$.
\\
\item $r=K_0+1,\cdots,N$. We fix $\epsilon=\frac{1}{2}$ and $R_{\frac{1}{2}}$ in (\ref{small}). By (\ref{compare1}) and (\ref{conv1}), there exists $\rho>0$, such that if
$|t_i-R(X_\triangle)|<\rho$, then for all $\delta x\in
B_{R_{\frac{1}{2}}}(0)\subset\mathbb{R}^n$,
\[
\lambda_r(D\tilde{u}_0(x_{t_i}+\delta
x))+1>e^{-CR_{\frac{1}{2}}}(\lambda_r(D\tilde{u}_0(x_{t_i}))+1)>e^{-CR_{\frac{1}{2}}}\frac{a_r}{2}>0
\]
\begin{eqnarray*}
\frac{1}{Vol(\triangle)}\int_{\mathbb{R}^n}\lambda_r(D\tilde{u}_0)e^{-w}dx+1&\ge&
\frac{1}{Vol(\triangle)}\int_{B_{R_{\frac{1}{2}}}(x_{t_i})}(\lambda_r(D\tilde{u}_0)+1)e^{-w}dx\\
&\ge&
e^{-CR_{\frac{1}{2}}}\frac{a_r}{2}\frac{1}{Vol(\triangle)}\int_{B_{R_{\frac{1}{2}}}(x_{t_i})}e^{-w}dx\\
&\ge& e^{-CR_{\frac{1}{2}}}\frac{a_r}{2}\frac{1}{2}>0
\end{eqnarray*}
\end{enumerate}
Now we prove the claim. We can rewrite (\ref{compare1}) using the
special form of $\tilde{u}_0$ (\ref{u0}).
\[
D\tilde{u}_0(x)=\sum_\alpha
\frac{e^{<p_\alpha,x>}}{\sum_{\beta}e^{<p_\beta,x>}}p_\alpha=\sum_\alpha
c_\alpha(x)p_\alpha
\]
Here the coefficients
\[
0\le
c_\alpha(x)=\frac{e^{<p_\alpha,x>}}{\sum_{\beta}e^{<p_\beta,x>}},\;
\sum_{\alpha=1}^Nc_\alpha(x)=1
\]
So
\[
\lambda_r(D\tilde{u}_0(x))+1=\sum_\alpha
c_\alpha(x)(\lambda_r(p_\alpha)+1)
=\sum_{\{\alpha:\lambda_r(p_\alpha)+1>0\}}c_\alpha(x)(\lambda_r(p_\alpha)+1)
\]
Since $\lambda_r(p_\alpha)+1\ge0$ is a fixed value, to prove the
claim, we only need to show the same estimate for $c_\alpha(x)$.
But now
\begin{eqnarray*}
c_\alpha(x_{t_i}+\delta
x)&=&\frac{e^{<p_\alpha,x_{t_i}>}e^{<p_\alpha,\delta
x>}}{\sum_{\beta}e^{<p_\beta,x_{t_i}>}e^{<p_\beta,\delta x>}}\le
e^{|p_\alpha|R}\cdot e^{
max_\beta|p_\beta|\cdot R}\frac{e^{<p_\alpha,x_{t_i}>}}{\sum_{\beta}e^{<p_\beta,x_{t_i}>}}\nonumber\\
&\le&
e^{CR}\frac{e^{<p_\alpha,x_{t_i}>}}{\sum_{\beta}e^{<p_\beta,x_{t_i}>}}=e^{CR}c_\alpha(x_{t_i})
\end{eqnarray*}
And similarly
\[c_\alpha(x_{t_i}+\delta
x)\ge e^{-CR}c_\alpha(x_{t_i})\]
So the claim holds and the proof is completed.
\end{proof}
\end{section}
\begin{section}{Example}
\begin{exmp}
$X_\triangle=Bl_p\mathbb{P}^2$. See the figure in Introduction.
$P_c=\frac{1}{4}(\frac{1}{3},-\frac{2}{3})$,
$-6P_c\in\partial\triangle$, so $R(X_\triangle)=\frac{6}{7}$.
\end{exmp}
\begin{exmp}
$X_\triangle=Bl_{p,q}\mathbb{P}^2$,
$P_c=\frac{2}{7}(-\frac{1}{3},-\frac{1}{3})$,
$-\frac{21}{4}P_c\in\partial\triangle$, so
$R(X_\triangle)=\frac{21}{25}$.
\end{exmp}
\begin{center}
\includegraphics[width=0.25\textwidth]{toric2.1}\hspace{0.2\textwidth}
\end{center}
\end{section}

Department of Mathematics, Princeton University, Princeton, NJ
08544, USA

E-mail address: chil@math.princeton.edu
\end{document}